\documentclass{amsart}
\usepackage{amssymb}
\usepackage[all]{xy}

\newcommand\boxb[1]{\square_b}

\numberwithin{equation}{section}

\newcommand\paperbody%
        {}

\newtheorem{lemma}{Lemma}
\newtheorem{proposition}{Proposition}

\newtheorem{theorem}{Theorem}
\newtheorem{non-theorem}{Non-Theorem}

\theoremstyle{remark}
\newtheorem{definition}{Definition}

\newcommand\coF{{}^{\mathcal{C}}\kern-2pt\Lambda}

\newcommand\cFTs{{}^{\Phi}\overline{T}\kern-1pt{}^*}

\newcommand\even{\text{even}}

\newcommand\pidl[2]{\mathcal{I}_{\text{H},+}^{#1}(#2)}

\newcommand\Tr{\operatorname{Tr}}

\newcommand\bTr{\overline{\operatorname{Tr}}}

\newcommand\rTr{\operatorname{Tr_{R}}}

\newcommand\Ch{\operatorname{Ch}}
\newcommand\Td{\operatorname{Td}}
\newcommand\com[1]{\overline{#1}}

\newcommand\ie{i\@.e\@. }

\newcommand\cL{\mathcal{L}}

\newcommand\bbC{\mathbb C}

\newcommand\bbQ{\mathbb Q}
\newcommand\bbR{\mathbb R}
\newcommand\bbS{\mathbb S}

\newcommand\bbZ{\mathbb Z}

\newcommand\CIc{{\mathcal{C}}^{\infty}_c}

\newcommand\CI{{\mathcal{C}}^{\infty}}

\newcommand\Diag{\operatorname{Diag}}

\newcommand\Ps[2]{\Psi^{#1}(#2)}
\newcommand\lPs[3]{\Psi^{#2}_{#1}(#3)}

\newcommand\HPs[2]{\Psi_{\text{H}}^{#1}(#2)}
\newcommand\IHPs[2]{\Psi_{\text{IH}}^{#1}(#2)}
\newcommand\IHlPs[3]{\Psi_{\text{IH},#1}^{#2}(#3)}

\newcommand\pnidl[2]{\mathcal{I}_{\text{H},\pm}^{#1}(#2)}

\newcommand\TpPs[2]{\Psi_{\text{Tp}}^{#1}(#2)}
\newcommand\ITpPs[2]{\Psi_{\text{ITp}}^{#1}(#2)}
\newcommand\ITplPs[3]{\Psi_{\text{ITp},#1}^{#2}(#3)}

\newcommand\cFNs{{}^{\Phi}\overline N\kern-1pt{}^*}

\newcommand\ind{\operatorname{ind}}

\newcommand\Hom{\operatorname{Hom}}
\newcommand\Id{\operatorname{Id}}

\newcommand\nul{\operatorname{null}}

\newcommand\pa{\partial}

\renewcommand\Re{\operatorname{Re}}

\newcommand\Fi{\text{F}}
\newcommand\Se{\text{S}}
\newcommand\Co{\text{C}}

\newcommand\Mand{\text{ and }}

\newcommand\Mat{\text{ at }}

\newcommand\Mon{\text{ on }}

\newcommand\Mover{\text{ over }}

\begin{document}
\title[Star products]
{Star products and local line bundles\\
\'Etoile-produits et fibr\'{e}s en droite locaux}

\author{Richard Melrose}
\address{Department of Mathematics, Massachusetts Institute of Technology}
\email{rbm@math.mit.edu}
\dedicatory{Dedicated to Louis Boutet de Monvel for the occasion
  of his sixtieth birthday\\
}
\begin{abstract} The notion of a local line bundle on a manifold,
  classified by 2-cohomology with real coefficients, is introduced. The
  twisting of pseudodifferential operators by such a line bundle leads to
  an algebroid with elliptic elements with real-valued index, given by a
  twisted variant of the Atiyah-Singer index formula. Using ideas of Boutet
  de Monvel and Guillemin the corresponding twisted Toeplitz algebroid on
  any compact symplectic manifold is shown to yield the star products of
  Lecomte and DeWilde (\cite{MR84g:17014}) see also Fedosov's construction in
  \cite{MR92k:58267}. This also shows that the trace on the star algebra
  is identified with the residue trace of Wodzicki (\cite{Wodzicki7}) and
  Guillemin  (\cite{Guillemin2}).

La notion de fibr\'{e}s en droite locaux sur une vari\'{e}t\'{e} 
diff\'{e}rentiable est introduite.  Ces derniers sont
classifi\'{e}s par la cohomologie r\'{e}elle de dimension 2.  Le twist
d'op\'{e}rateurs pseudodiff\'{e}rentiels par de tels fibr\'{e}s en droite donne
lieu \`{a} un alg\'{e}bro\"{i}de contenant des \'{e}l\'{e}ments elliptiques 
dont l'indice \`{a} valeur dans les r\'{e}els est donn\'{e} par une variante de
la formule de l'indice d'Atiyah et Singer.  En utilisant des id\'{e}es de 
Boutet de Monvel et Guillemin, on montre que, sur toute vari\'{e}t\'{e} 
symplectique compacte, il est possible d'obtenir le produit \'{e}toil\'{e} de
Lecomte et DeWilde \cite{MR84g:17014} (voir aussi la construction de 
Fedosov \cite{MR92k:58267})  \`{a} partir de l'alg\'{e}bro\"{i}de 
associ\'{e} au twist des op\'{e}rateurs de Toeplitz.  Cela \'{e}tablit du
m\^{e}me coup que la trace d\'{e}finie sur cette alg\`{e}bre \'{e}toil\'{e}e 
peut \^{e}tre identifi\'{e}e avec la trace r\'{e}siduelle de Wodzicki 
\cite{Wodzicki7} et Guillemin \cite{Guillemin2}.
\end{abstract}

\maketitle

\tableofcontents

\section*{Introduction}


If $M$ is a symplectic manifold, Lecomte and DeWilde (\cite{MR84g:17014},
see also Fedosov's construction, \cite{MR92k:58267}) showed that $M$ carries a
star product. That is, the space of formal power series in a parameter,
$t,$ with coefficients being smooth functions on $M,$ carries an associative
product
\begin{multline}
\star:\CI(M)[[t]]\times\CI(M)[[t]]\longrightarrow \CI(M)[[t]],\\
(\sum\limits_{j\ge0}a_jt^j)\star(\sum\limits_{l\ge0}b_lt^l)=
(\sum\limits_{k\ge0}c_kt^k),\
c_k=\sum\limits_{j+l\le k}B_{k,j,l}(a_j,b_l)
\label{spallb.13}\end{multline}
where each $B_{k,j,l}$ is a bilinear differential operator and 
\begin{equation}
B_{0,0,0}(\alpha ,\beta )=\alpha \beta ,\ B_{1,0,0}(\alpha ,\beta)=\{\alpha
,\beta \}.
\label{spallb.14}\end{equation}
Here the second term is the Poisson bracket on the symplectic manifold;
note that the normalization of the non-zero coefficient is arbitrary, since
it can be changed by scaling the formal variable $t.$

The star product is \emph{pure} if $B_{k,j,l}=\tilde B_{k-j-l}$ only
depends on the `change of order'. This means that it can be written 
\begin{equation}
a\star b=\left(\sum\limits_{l}t^l\tilde B_l(x,y)\right)a(x)b(y)\big|_{x=y}.
\label{spallb.15}\end{equation}

Boutet de Monvel and Guillemin had an alternate approach to `formal
quantization' in this sense, but it was limited to the \emph{pre-quantized}
case, in which the symplectic form is assumed to be a non-vanishing
multiple of an integral class and hence arises from the curvature of a line
bundle. The corresponding circle bundle is a contact manifold and the
quantization of $M$ arises from the choice of an $\bbS$-invariant Toeplitz
structure on this contact manifold; note that this gives much more than the
star product. This construction is reviewed below and extended to the
general case using the notion of a local line bundle. In a certain sense
the construction here is intermediate between that of Boutet de Monvel and
Guillemin and that of Fedosov (which for the sake of brevity is not
discussed) but still gives more than the latter in so far as the star
algebra is shown to be the quotient of an algebroid of Toeplitz-like
kernels near the diagonal, where composition is only restricted by the
closeness of the support to the diagonal. The quotient is by the
corresponding algebroid of smoothing operators. This carries the usual
trace functional and the unique (normalized) trace on the star algebra is
shown to arise as a residue trace in this way, with the trace-defect
formula used to prove the homotopy invariance of the index as in
\cite{mms3}.

\paperbody

\section{Toeplitz operators}

Under the assumption that $(M,\omega)$ is a symplectic manifold with
$[\omega ]\in H^2(M,\bbZ)$ an integral class, Guillemin, in
\cite{MR96g:58066}, exploiting his earlier work with Boutet de Monvel
(\cite{BoutetdeMonvel-Guillemin1}) used the existence of a Toeplitz algebra
on the circle bundle with curvature $\omega$ to construct a star product on
$M.$ This construction is first sketched and then extended to the
non-integral case.

Let $L$ be an Hermitian line bundle over $M$ with unitary connection having
curvature $\omega/2\pi;$ this exists in virtue of the assumed integrality
of $\omega.$ Let $Z$ be the circle bundle of $L.$ The connection on $L$ induces
a connection 1-form, $\eta\in\CI(Z;\Lambda ^1),$ on $Z,$ fixed by the two
conditions 
\begin{enumerate}
\item $\eta (\pa_{\theta})=1$ for the derivative of the circle action and
\item For each $p\in Z,$ $\eta_p$ is normal to any local section of $Z$
  over $M$ which is covariant constant, at $p,$ as a section of $L.$
\end{enumerate}
It follows that $\eta$ is $\bbS$-invariant and $d\eta =\omega$ is a basic
form. Since $\omega$ is assumed to be symplectic, $\eta$ is a contact form
on $Z.$

Any contact manifold, $Y,$ carries a natural space of `Heisenberg'
pseudodifferential operators, see \cite{Beals-Greiner1}, \cite{Taylor3} and
\cite{Geller1} and \cite{HHH2}. If the contact manifold is compact these
form an algebra, in general the properly supported elements form an
algebra. Let us denote by $\HPs0Y$ the space of Heisenberg
pseudodifferential operators of order $0$ on $Z.$ This also has two ideals,
corresponding to the two orientations of the contact bundle, namely the
upper and lower Hermite ideals $\pnidl0Y\subset \HPs0Y.$ The intersection
of these ideals is the space of (properly supported) smoothing operators.

Although working more from the point of view of complex Lagrangian
distributions, Boutet de Monvel and Guillemin introduced the notion of a
`quantized contact structure' which is the choice of a generalized Szeg\H o
projector $P\in\pidl0Y.$ Assuming $Y$ to be compact, $P^2=P,$ otherwise it
is properly supported and such that $P^2-P$ is a smoothing operator. The
defining property of such a projector is that its symbol, in the sense of
the Heisenberg algebra, should be the field of projections, one for each
point of $Y,$ onto the null space of the field of harmonic oscillators
arising from the choice of a compatible almost complex structure. It is
shown in \cite{MR2000a:58062} that the set of components of such
projections is mapped onto $\bbZ$ by the relative index, once a base point
is fixed.

Having chosen such a projector, the associated space of Toeplitz operators
consists of the compressions of pseudodifferential operators (or Heisenberg
pseudodifferential operators) to the range of $P,$ \ie the operators 
\begin{equation}
\TpPs0Y=\left\{PAP;A\in\Ps0Y\right\}.
\label{spallb.1}\end{equation}
If $Y$ is compact this is again an algebra; otherwise the properly
supported elements form an algebra if the smoothing operators are
appended. In either case, the quotient by the corresponding algebra of
smoothing operators is the `Toeplitz full symbol algebra'  
\begin{equation}
\TpPs0Y/\TpPs{-\infty}Y\simeq\CI(Y)[[\rho ]].
\label{spallb.2}\end{equation}
Here the formal power series parameter is the inverse of the homogeneous,
length, function on the contact line bundle over $Y.$

Returning to the case that the contact manifold, now $Z,$ is a circle
bundle with $\bbS$-invariant contact structure we may choose the projection
$P$ to be $\bbS$-invariant and then consider the subspace of
$\bbS$-invariant Toeplitz operators 
\begin{equation}
\ITpPs0Z=\left\{A\in\TpPs0Z:T_\theta^*A=AT_\theta^*\ \forall\ \theta
\in\bbS\right\}.
\label{spallb.3}\end{equation}
For this subalgebra 
\begin{equation}
\ITpPs0Z/\ITpPs{-\infty}Z\simeq\CI(M)[[\rho ]].
\label{spallb.4}\end{equation}

\begin{theorem}\label{spallb.5} (Guillemin \cite{MR96g:58066}) If
  $(M,\omega)$ is a compact integral symplectic manifold then $\ITpPs0Z,$
  the $\bbS$-invariant part of the Toeplitz operators for a choice of
  $\bbS$-invariant generalized Szeg\H o projector on the circle bundle of
  an Hermitian line bundle with curvature $\omega /2\pi,$ is an algebra and
  this algebra structure induces a star product on $M$ through \eqref{spallb.4}.
\end{theorem}

\noindent In case $M$ is non-compact, but still with integral symplectic
structure, the choice of a properly supported $\bbS$-invariant projection,
up to smoothing, leads to the same result for the properly supported
Toeplitz operators, with properly supported smoothing operators appended.

For the proof see \cite{BoutetdeMonvel-Guillemin1} or \cite{HHH2}.

\section{Closed 1-forms}

As a slight guide to the discussion of local line bundles below we first
discuss the analogous `geometric model' for 1-dimensional real
cohomology. This result is not used anywhere below.

As is well-known, the closed 1-forms inducing integral 1-dimensional
cohomology classes on a manifold $M$ can be realized in terms of functions
into the circle. Thus,
\begin{equation}
a\in\CI(M;\bbS)\longleftrightarrow\frac1{2\pi i}a^{-1}da\in\CI(M;\Lambda^1)
\label{spallb.52}\end{equation}
is an isomorphism onto the real closed integral 1-forms. One can get a closely
related realization of the cohomology with real coefficients in terms of
`local circle functions' on $M.$

\begin{definition}\label{spallb.53} A \emph{local circle function} on $M$ is a
  smooth map defined on a neighbourhood of the diagonal in $M^2,$
  $A\in\CI(W;\bbS),$ $W\subset M^2$ open, $\Diag\subset W,$ such that 
\begin{equation}
A(x,y)A(y,z)=A(x,z)\ \forall\ (x,y,z)\in V,
\label{spallb.54}\end{equation}
where $V$ is some neighbourhood of the triple diagonal in $M^3.$
\end{definition}

In fact we will only consider germs of such functions at the diagonal, \ie
identify two such functions if they are equal in some neighbourhood of the
diagonal. If $a\in\CI(M;\bbS)$ then $A(x,y)=a(x)a^{-1}(y)$ satisfies
\eqref{spallb.54}. Setting $x=y=z$ in \eqref{spallb.54} shows that
$A\big|_{\Diag}=1.$ Similarly $A(x,y)A(y,x)=1$ near the diagonal. If
$U\subset M$ is a small open set so that $U\times U\subset W$ and $U\times
U\times U\subset V$ then choosing $p\in U$ and setting $a_p(x)=A(x,p),$
$x\in U,$ gives $a_p:U\longrightarrow \bbS$ and $A(x,y)=a_p(x)a_p^{-1}(y)$
on $U\times U.$ Thus $A$ does define such a map, $a,$ locally. Changing the
base point $p$ to another $q\in U$ changes $a_p$ to
$a_q(x)=A(x,q)=A(x,p)A(p,q)=A(p,q)a_p(x),$ \ie only by a multiplicative
constant. It follows that the 1-form 
\begin{equation}
\begin{gathered}
\alpha =\frac1{2\pi i}A^{-1}d_xA\Mat\Diag\simeq M\\
=\frac1{2\pi i}a_p^{-1}da_p
\in\CI(M;\Lambda^1)
\end{gathered}
\label{spallb.55}\end{equation}
is well-defined on $U$ independently of the choice of $p\in U$ and hence is
globally well-defined on $M.$ From the local identification with
$\frac1{2\pi 1}a_p^{-1}da_p$ it is clearly closed. Furthermore, the
vanishing of $\alpha$ implies that $a_p$ is locally constant and hence
$A\equiv 1$ near the diagonal. Thus we have proved

\begin{proposition}\label{spallb.56} The group of germs at the diagonal of
  local circle functions is isomorphic to the space of closed 1-forms.
\end{proposition}

Similarly, if we consider real functions, $\beta \in\CI(W,\bbR),$ defined
near the diagonal which satisfy the additivity condition  
\begin{equation}
\beta (x,y)+\beta (y,z)=\beta (x,z)\Mon V
\label{spallb.57}\end{equation}
and identify the local circle functions $A$ and $e^{2\pi i\beta }A$ then we
arrive at a geometric realization of $H^1(M,\bbR).$

\section{Local line bundles}

The problem with extending Theorem~\ref{spallb.5} to the general case is,
of course, that in the non-integral case there can be no line bundle over
$M$ with curvature the symplectic form. Nevertheless there is a `virtual
object' which plays at least part of the same role. This is closely related
to, but rather simpler than, Fedosov's theorem on the classification of
star products, up to isomorphism, in terms of $H^2(M;\bbR).$

We give a `geometric realization' of $H^2(M;\bbR)$ on any manifold,
possibly with corners. The construction here of `local line bundles' over
$M$ is related to ideas of Murray concerning bundle gerbes (\cite{Murray1})
and more particularly to the discussion of extensions of Azumaya bundles in
\cite{mms3}.

Consider the diagonal  
\begin{equation}
\Diag=\{(z,z)\in M^2\}
\label{17.2.2004.3}\end{equation}
which is naturally diffeomorphic to $M$ under either the left or right
projection from $M^2$ to $M.$ Similarly consider the triple diagonal 
\begin{equation}
\Diag_3=\{(z,z,z)\in M^3\}.
\label{17.2.2004.4}\end{equation}
There are three natural projections from $M^3$ to $M^2$ which we label
$\pi_{F},$ $\pi_{S}$ and $\pi_{C}$ (for `F'irst, `S'econd and `C'entral or
`C'omposite): 
\begin{equation}
\begin{gathered}
\pi_{F}:M^3\ni(x,y,z)\longmapsto (y,z)\in M^2,\\
\pi_{S}:M^3\ni(x,y,z)\longmapsto (x,y)\in M^2\Mand\\
\pi_{C}:M^3\ni(x,y,z)\longmapsto (x,z)\in M^2.
\end{gathered}
\label{17.2.2004.5}\end{equation}

\begin{definition}\label{17.2.2004.1} A \emph{local line bundle} over a
manifold $M$ is a (complex) line bundle $\cL$ over a neighbourhood of
$\Diag\subset M^2$ together with a smooth `composition' isomorphism over a
neighbourhood of $\Diag_3\subset M^3$
\begin{equation}
H:\pi^*_{\Se}\cL\otimes \pi_{\Fi}^*\cL\longrightarrow \pi_{\Co}^*\cL
\label{17.2.2004.2}\end{equation}
with the associativity condition that for all $(x,y,z,t)$ sufficiently
close to the total diagonal in $M^4,$ the same map
\begin{equation}
\cL_{(x,y)}\otimes \cL_{(y,z)}\otimes \cL_{(z,t)}\longrightarrow \cL_{(x,y)}
\label{17.2.2004.8}\end{equation}
arises either by first applying $H$ in the left two factors and then on the
composite, or first in the right two factors and then in the composite, \ie
the following diagramme commutes:
\begin{equation}
\xymatrix{
&\cL_{(x,z)}\otimes \cL_{(z,t)}\ar[dr]^{H_{(x,z,t)}}\\
\cL_{(x,y)}\otimes \cL_{(y,z)}\otimes
\cL_{(z,t)}\ar[ur]^{H_{(x,y,z)}\otimes\Id}
\ar[dr]_{\Id\otimes H_{(y,z,t)}}&& \cL_{(x,t)}.
\\
&\cL_{(x,y)}\otimes \cL_{(y,t)}\ar[ur]_{H_{(x,y,t)}}}
\label{17.2.2004.9}\end{equation}
\end{definition}
\noindent We will really deal with germs at the diagonal of these objects.

Over the triple diagonal itself $H$ necessarily gives an isomorphism 
\begin{equation}
H_{\Diag}:\cL_{\Diag}\otimes \cL_{\Diag}\longrightarrow \cL_{\Diag}.
\label{17.2.2004.6}\end{equation}
If $e_z\not=0$ is an element of $\cL_{(z,z)}$ then $H_{\Diag}$ maps
$e_z\otimes e_z$ to $ce_z$ for some $0\not=c\in\bbC.$ Thus, corresponding
to the two square-roots of $c$ there are exactly two local sections of
$\cL_{\Diag}$ such that $H(c\otimes c)=c.$ On the other hand $\cL_{(x,x)}$ acts
on the right on each $\cL_{(t,x)}$ for $t$ sufficiently close to $x$ as the
space of homomorphisms. The associativity condition \eqref{17.2.2004.9}
means that $\cL_{(x,x)}$ is identified with the linear space of homomorphisms
on $\cL_{(y,x)}$ in a way consistent with its product. Thus either $e$ or
$-e$ must be locally the identity. Hence $\Id$ must exist as a global section
so there is a canonical identification 
\begin{equation}
\cL_{\Diag}\longrightarrow M\times\bbC
\label{17.2.2004.10}\end{equation}
consistent with the action of $\cL_{\Diag}$ on the left or right on $\cL$
through $H.$

Now, consider local trivializations of $\cL.$ Choosing a sufficiently
fine open cover $\{U_i\}$ of $M,$ the products $U_i\times U_i$ give an open
cover of $\Diag\subset M^2$ and are contained in a given neighbourhood of
the diagonal. Thus, for some such open cover, $\cL$ is defined over each
$U_i\times U_i.$ Choose a point $p_i\in U_i$ and consider the bundles 
\begin{equation}
L_{i,p_i}=\cL\big|_{U_i\times\{p_i\}},\ R_{i,p_i}=\cL\big|_{\{p_i\}\times U_i}.
\label{17.2.2004.12}\end{equation}
The composition law $H$ gives an identification 
\begin{equation}
\cL\overset{H}\equiv L_{i,p_i}\otimes R_{i,p_i}\Mover U_i\times U_i
\label{17.2.2004.13}\end{equation}
and also an identification 
\begin{equation}
R_{i,p_i}=L_{i,p_i}^{-1},
\label{17.2.2004.16}\end{equation}
since over the diagonal $\cL$ has been canonically trivialized.

\begin{lemma}\label{17.2.2004.11} Any local line bundle over $M$ has a
multiplicative connection, \ie a connection $\nabla$ such that if $u$ is a
local section of $\cL$ near $(x,y)$ with $\nabla u=0$ at $(x,y)$ and $v$ is a
local section of $\cL$ near $(y,z)$ with $\nabla v=0$ at $(y,z)$ then
$H(u,v)$ is locally constant at $(x,y,z).$ Similarly $\cL$ has a
multiplicative unitary structure, so 
\begin{equation}
|H(u,v)|=|u||v|
\label{spallb.32}\end{equation}
and has a multiplicative Hermitian connection.
\end{lemma}

\begin{proof} Using an open cover as described above, choose $p_i\in U_i$
  for each $i$ and a connection $\nabla_i$ on $L_{i,p_i}.$ If we make a
  different choice, $q_i,$ of point in $U_i$ then the composition law $H$
  gives an identification of $L_{i,q_i}=\cL\big|_{U_i\times\{q_i\}}$ with
  $L_{i,p_i}\otimes \cL_{p_i,q_i}.$ The second factor is a fixed complex
  line, so a connection on $L_{i,p_i}$ induces, through $H,$ a connection on
  $L_{i,q_i}.$ Now, choose a partition of unity subordinate to the cover,
  $\rho _i\in\CIc(U_i)$ with $\sum\limits_{i}\rho _i=1.$ We shall modify
  the connection on $L_{i,p_i}$ and replace it by 
\begin{equation}
\nabla=\sum\limits_{j}\rho _i\nabla_j\Mon L_{i,p_i}\Mover U_i.
\label{17.2.2004.14}\end{equation}
Here we use the fact that over $U_{ij}=U_i\cap U_j$ the
  line bundles $L_{i,p_i}$ and $L_{j,p_j}$ are identified by $H$ after
  tensoring with the fixed line $\cL_{p_i,p_j},$ as discussed above, so
  $\nabla_j$ is well-defined on $L_{i,p_i}$ over $U_{ij}$ which contains
  the support of $\rho _j$ in $U_i.$ Directly from the definition, this new
  connection is consistent with the identification of $L_{i,p_i}$ and
  $L_{j,p_j}$ over $U_{ij}.$

Now the connection on $\cL$ over $U_i\times U_i$ induced by taking the dual
connection to $\nabla$ on $R_{i,p_i},$ using the identification
\eqref{17.2.2004.16}, and then the tensor product connection on $\cL$ using
\eqref{17.2.2004.13} is independent of $i.$ That is, it is a global
connection on $\cL$ and from its definition has the desired product property.

The same approach allows one to define a multiplicative Hermitian structure
by taking as Hermitian structure $\langle \cdot,\cdot\rangle_i$ on each
$L_{i,p_i}.$ This induces Hermitian structures on the inverses
$L^{-1}_{i,p_i}$ and hence on $\cL$ over $U_i\times U_i.$ Then if $\rho _i$
is a partition of unity subordinate to the cover, the inner product on $\cL$
over $U_i\times U_i$ 
\begin{equation}
\langle u,v\rangle =\sum\limits_{i}(\rho _i\times\rho _i)\langle u,v\rangle _i
\label{spallb.33}\end{equation}
is consistent with the inner products over the other $U_j\times U_j.$
Unitary metrics on each of the $L_{i,p_i}$ then induce a connection on
$\cL$ which is both multiplicative and unitary, \ie is consistent with the
Hermitian structure.
\end{proof}

\begin{proposition}\label{17.2.2004.17} The left curvature of a product
Hermitian connection, \ie the restriction of the curvature at the diagonal
to the left tangent space, is an arbitrary real closed $2$-form on $M$ 
lying in a fixed class in $H^2(M,\bbR)$ determined by the local line bundle
and product Hermitian structure. Two local line bundles are isomorphic in
some neighbourhood of the diagonal under a unitary isomorphism intertwining the
product structures if and only if the left curvatures define the same
cohomology class; all cohomology classes arise in this way.
\end{proposition}

\begin{proof} Two product connections on a fixed local line bundle differ
  (in a small neighbourhood of the diagonal where they are both defined) by
  a $1$-form $i\alpha.$ The multiplicative condition on the connections
  implies that
\begin{equation}
\alpha _{(x,y)}(v,w)=\alpha _{(x,z)}(v,u)+\alpha _{(y,z)}(u,w)
\label{spallb.34}\end{equation}
for all points $(x,z,y)$ in a small enough neighbourhood of the triple
diagonal and all $v\in T_xX,$ $u\in T_yX$ and $w\in T_zX.$ In particular
$\alpha _{(x,x)}=0$ and $\alpha_{(x,y)}(u,w)+\alpha _{(y,x)}(w,v)=0$ so the
$1$-form defined locally on $M$ by  
\begin{equation}
\beta _x(v)=\alpha _{(x,y)}(v,0),\ v\in T_xX,
\label{spallb.35}\end{equation}
for $y$ fixed close to $x$ is actually globally well-defined and satisfies 
\begin{equation}
\alpha=\pi_L^*\beta -\pi _R^*\beta.
\label{spallb.36}\end{equation}
Similarly the curvature of a product connection is locally the curvature of
a connection on $L_{i,p_i}\boxtimes L^{-1}_{i,p_i}$ coming from a
 connection on $L_{i,p_i}.$ It is therefore locally of the form
$\pi_L^*\omega -\pi^*\omega$ for a closed $2$-form on $M.$ Again it
follows that the $2$-form $\omega$ is well-defined, as the restriction of
the curvature to left tangent vectors at the diagonal.

Thus it remains to show that any real closed $2$-form, $\omega,$ on $M$
arises this way. Consider a good cover of $M,$ so each of the open sets
$U_i$ and all of their non-trivial intersections are contractible. Then on
each $U_i$ there exists a smooth $1$-form, $\alpha _i\in\CI(U_i;\Lambda
^1)$ such that
\begin{equation}
\omega =d\alpha _i\Mon U_i.
\label{spallb.37}\end{equation}
On non-trivial overlaps there exists a smooth function $\phi _{ij}$ such
that 
\begin{equation}
\alpha _i-\alpha _j=d\phi _{ij}\Mon U_i\cap U_j.
\label{spallb.38}\end{equation}
It follows that on non-trivial triple intersections
\begin{equation}
\phi _{ij}+\phi _{jk}+\phi _{ki}=\phi_{ijk}\text{ is constant on }U_i\cap
U_j\cap U_k.
\label{spallb.39}\end{equation}
Now, consider $L_i$ which is the trivial line bundle over $U_i,$ with the
connection $d+i\alpha_i$ which is unitary for the standard Hermitian
structure. Over the open neighbourhood of the diagonal
\begin{equation}
U=\bigcup\limits_{i}(U_i\times U_i)
\label{spallb.40}\end{equation}
use the product isomorphism $e^{i\phi_{ij}}\times e^{-i\phi_{ij}}$ to
identify $L_i\boxtimes L_i^{-1}$ with $L_j\boxtimes L_j^{-1}$ over the
intersection $(U_i\cap U_j)\times(U_i\cap U_j).$ The constancy of the $\phi
_{ijk}$ means that these unitary isomorphisms satisfy the cocycle condition
on triple overlaps, so this gives a well-defined bundle $\cL$ over $U.$
That this is a local line bundle follows immediately from its definition
and the connections patch to give a global unitary connection with
curvature $\omega.$
\end{proof}

Thus the collection of Hermitian local line bundles modulo unitary
multiplicative isomorphisms in some neighbourhood of the diagonal is
identified with $H^2(M;\bbR)$ and the collection of Hermitian local line
bundles with unitary product connections modulo isomorphisms identifying
the connections is identified with the space of real closed 2-forms on $M.$
By extension from the standard case we call the cohomology class 
\begin{equation}
\frac{\omega }{2\pi}\in H^2(M,\bbR)
\label{spallb.41}\end{equation}
the first Chern class of $\cL$ and $\exp(\omega /2\pi)\in
H^{\even}(M;\bbR)$ its Chern character.

\section{Atiyah-Singer index formula}

The Atiyah-Singer formula expresses the analytic index, the difference of
the dimension of the null space and the null space of the adjoint, for an
elliptic pseudodifferential operator in terms of topological data
determined by the principal symbol. Namely if $A\in\Ps m{X;E,F}$ is a
pseudodifferential operator acting between sections of the two (complex)
vector bundles $E$ and $F$ then its principal symbol $\sigma
_m(A)\in\CI(S^*X;\hom(E,F)\otimes N_m)$ defines a homomorphism between the
lifts of $E$ and $F$ to the cosphere bundle, up to a positive diagonal
factor; ellipticity of $A$ is by definition equivalent to invertibility of
this homomorphism. This in turn fixes a compactly supported K-class on $T^*X,$ or
equivalently a K-class for the radial compactification $\com{T^*X}$ of the
cotangent bundle relative to its boundary, the cosphere bundle 
\begin{equation}
[\sigma _m(A)]\in K(\com{T^*X},S^*X].
\label{spallb.16}\end{equation}
The Atiyah-Singer formula is
\begin{multline}
\ind(A)=\dim\nul(A)-\dim\nul(A^*)\\
=\Tr(AB-\Id_F)-\Tr(BA-\Id_E)=\int_{T^*X}\Ch(\sigma _m(A))\Td(X).
\label{spallb.17}\end{multline}
Here $B\in\Ps{-m}{X;F,E}$ is a parametrix for $A,$ so is such that
$AB-\Id_F,$ $BA-\Id_E$ are smoothing operators, $\Ch:K(M,\pa M)\longrightarrow
H^*(M,\pa M)$ is the Chern character and $\Td(X)$ is the Todd class
of the cotangent bundle.

We wish to generalize this formula to include twisting by a local line
bundle as discussed above. To do so, recall the definition of the space of
pseudodifferential operators of order $m.$ If $\Diag\subset X^2$ is the
diagonal in the product then in terms of Schwartz' kernels 
\begin{equation}
\Ps m{X;E,F}=I^m(X^2,\Diag;\Hom(E,F)\otimes\Omega _R).
\label{spallb.18}\end{equation}
Here $I^m(M,Y;G)$ is the space of conormal distributions of order $m,$
introduced explicitly by H\"ormander in \cite{Hormander3} for any embedded
submanifold $Y$ of a manifold $M$ and any vector bundle $G$ over $M.$ In
the particular case \eqref{spallb.18}, $\Hom(E,F)\equiv \pi_L^*F\otimes
\pi_R^*E'$ is the `big' homomorphism bundle over the product and $\Omega
_R=\pi_R^*\Omega$ is the right-density bundle, allowing invariant
integration; here $\pi_L,\pi_R:X^2\longrightarrow X$ are the two
projections. The symbol $\sigma _m(A)$ for $A\in\Ps m{X;E,F}$ is then the
leading asymptotic term in the Fourier transform, in directions transversal
to the diagonal, of the kernel and is naturally identified with a section
of the `little' homomorphism bundle $\hom(E,F)=\Hom(E,F)|_{\Diag}$ lifted to
the conormal bundle of $\Diag,$ which is to say the cotangent bundle of $X.$ 

As in \cite{mms3} we can general this space of kernels by twisting with a vector
bundle, even if this bundle is only defined in a neighbourhood of the
diagonal. In this way we will obtain a space of `kernels' with supports in
a sufficiently small neighbourhood of $\Diag\subset X^2.$ In \cite{mms3}
this was considered in the case of the homorphism bundle for a projective
vector bundle and also for local line bundles which are $N$th roots of line
bundles over $X.$ Here we can allow the more general case of a local line
bundle as discussed above and define 
\begin{equation}
\lPs {\cL,\epsilon}m{X;E,F}=
I^m(B_\epsilon,\Diag;\Hom(E,F)\otimes\cL\otimes\Omega _R)
\label{spallb.19}\end{equation}
where $B_\epsilon$ is a sufficiently small neighbourhood of the diagonal
over which the local line bundle $\cL$ exists. For definiteness, and
because it is related to Weyl quantization, we will take $B_\epsilon$ to
be the points of $X^2$ distant less than $\epsilon$ from the diagonal with
respect to a metric on $X$ on the two factors of $X^2.$ 
  
In general the elements if $\lPs{\cL,\epsilon}m{X;E,F}$ do not compose
freely as do those of $\Ps m{X;E,F};$ rather it is necessary for the
supports to be sufficiently close to the diagonal. Thus suppose $E,F$ and
$G$ are three vector bundles over $X.$ Composition  
\begin{equation}
\Ps m{X;F,G}\cdot\Ps{m'}{X;E,F}\subset\Ps{m+m'}{X;E,G}
\label{spallb.20}\end{equation}
in the standard case, reduces to a push-forward operation on the
kernels. Namely 
\begin{multline}
I^m(X^2,\Diag,\Hom(F,G)\otimes\Omega _R)\times
I^{m'}(X^2,\Diag;\Hom(E,F)\otimes \Omega _R)\\
\longrightarrow
I^{m+m'}(X^2,\Diag,\Hom(E,G)\otimes\Omega _R).
\label{spallb.21}\end{multline}

This push-forward result, and correspondingly the composition
\eqref{spallb.20} can be localized on $X^2$ in each factor. Thus,
localizing away from the diagonal gives a smooth term and this results in a
smoothing operator. Localization near a point on the diagonal in either
factor allows the vector bundles to be trivialized and then the result
reduces to the scalar case for open sets in Euclidean space. There, or even
globally, the elements of $I^m(X^2;\Diag),$ which are the \emph{classical}
(so polyhomogeneous) conormal distributions may be approximated by smooth
functions within the somewhat larger class of conormal distributions `with
bounds' (\ie of type $1,0.)$ In the smoothing case, \ie for $m=m'=-\infty,$
\eqref{spallb.21} becomes 
\begin{multline}
\CI(X^2;\Hom(F,G)\otimes\Omega _R)\times\CI(X^2;\Hom(E,F)\otimes\Omega _R)\\
\longrightarrow \CI(X^3;\pi_R^*\otimes\pi_M^*F'\otimes
\pi_M^*F\otimes\pi_R^*E\otimes\pi_M^*\Omega \otimes\pi_R^*\Omega )\\
\longleftrightarrow\CI(X^2;\Hom(E,F)\otimes\Omega _R).
\label{spallb.22}\end{multline}
Here the three projections $\pi_L,\pi_M,\pi_R:X^3\longrightarrow X$ are
used and for the second map the pairing between $E$ and $E'$ leads to a
density in the middle factor which is integrated out. Thus
\eqref{spallb.22} is a more explicit, and invariant, version of the
composition formula
\begin{equation}
A\circ B(x,y)=\int_X A(x,z)B(z,y)dz.
\label{spallb.23}\end{equation}
In particular to extend \eqref{spallb.20} to the kernels in
\eqref{spallb.19} it is only necessary to see that the smooth composition
makes sense as in \eqref{spallb.22} since the singularities of the kernels
behave exactly as in the standard case. In fact in the presence of a local
line bundle $\cL,$ \eqref{spallb.22} is replaced by
\begin{multline}
\CI(B_{\epsilon };\Hom(F,G)\otimes\cL\otimes\Omega
_R)\times\CI(B_{\epsilon '};\Hom(E,F)\otimes\cL\otimes\Omega _R)\\
\longrightarrow \CI(\pi_S^{-1}B_\epsilon\cap\pi_F^{-1}B_{\epsilon '}
;\pi_R^*G\otimes\pi_M^*F'\otimes
\pi_M^*F\pi_S^*\cL\otimes\pi_F^*\cL
\otimes\pi_R^*E\otimes\pi_M^*\Omega \otimes\pi_R^*\Omega )\\
\longleftrightarrow\CI(B_\eta ;\Hom(E,F)\otimes\cL\Omega _R).
\label{spallb.22a}\end{multline}
Here $\pi_S,\pi_F,\pi_C:X^3\longrightarrow X^2$ are the projections from
\eqref{17.2.2004.5} and give, over a sufficiently small neighbourhood of
the triple diagonal, an identification
\begin{equation}
H:\pi_S^*\cL\otimes\pi_F^*\cL\longrightarrow \pi_C^*\cL.
\label{spallb.24}\end{equation}
Notice that if $B_\epsilon$ and $B_\epsilon'$ are sufficiently small
neighbourhoods of the diagonals then
$\pi_S^{-1}B_\epsilon\cap\pi_F^{-1}B_{\epsilon '}$ is indeed an arbitrarily
small neigbourhood of the triple diagonal, projecting under $\pi_C$ to a
neighbourhood of the diagonal $B_\epsilon,$ small with $\epsilon
+\epsilon'.$

Thus, when the product \eqref{spallb.22a} is localized near a point on the
diagonal in each factor, and these points can always be taken to be the
same, $\cL$ reduces to $\pi_L^*L\otimes \pi_R^*L^{-1}$ and it becomes
\eqref{spallb.22} with $E,F$ and $G$ all replaced by $E\otimes L,$
$F\otimes L$ and $G\otimes L.$ Thus indeed we arrive at the restricted, but
associative, product 
\begin{equation}
\lPs {\cL,\epsilon }m{X;F,G}\cdot\lPs{\cL,\epsilon '}{m'}{X;E,F}\subset
\lPs{\cL,\eta }{m+m'}{X;E,G}
\label{spallb.25}\end{equation}
for $\epsilon +\epsilon '$ small compared to $\eta.$ This product can
still be written as in \eqref{spallb.23} but with the associative product
$H$ giving the pairing on $\cL.$ Notice that the symbol is well-defined, as
it is in the standard case, and since it is just a section of the
restriction to the diagonal of the bundle it leads to a short exact
sequence 
\begin{equation}
\xymatrix{\lPs{\cL,\epsilon }{m-1}{X;E,F}\ar[r]&
\lPs{\cL,\epsilon }m{X;E,F}\ar[r]^-{\sigma _m}&
\CI(S^*X;\hom(E,F)\otimes N_m)}
\label{spallb.26}\end{equation}
in which the twisting local bundle $\cL$ does not appear in the symbol.

\begin{theorem}\label{spallb.27} If $A\in\lPs{\cL,\epsilon }m{X;E,F}$ is
  elliptic, in the sense that $\sigma _m(A)$ is invertible, and $\epsilon
  >0$ is sufficiently small then there is a parametrix
  $B\in\lPs{\cL,\epsilon '}{-m}{X;F,E},$ for any $\epsilon'>0$ sufficiently
  small, such that $AB-\Id$ and $BA-\Id$ are smoothing operators and then
  the index
\begin{equation}
\ind(A)=\Tr[A,B]\in\bbR
\label{spallb.28}\end{equation}
is independent of the choice of $B,$ is log-multiplicative for elliptic
  operators  
\begin{equation}
\ind(AA')=\ind(A)+\ind(A')
\label{spallb.29}\end{equation}
is homotopy invariant under elliptic deformations, is additive under direct
  sums and so defines an additive map 
\begin{equation}
\ind_{\cL}:K(\com{T^*X};S^*X)\longrightarrow \bbR,\
  \ind(A)=\ind_{\cL}([E,F,\sigma _m(A)]),
\label{spallb.30}\end{equation}
which is given by a variant of the Atiyah-Singer formula 
\begin{equation}
\ind_{\cL}(a)=\int_{T^*X}\Ch(a)\Td(X)\exp(\omega /2\pi)
\label{spallb.31}\end{equation}
where $\omega/2\pi\in H^2(X;\bbR)$ is the first Chern class of the local line
  bundle $\cL.$
\end{theorem}

\begin{proof} The proof of this result is essentially the same as that of
  the twisted index theorem in \cite{mms3}; we recall the steps.

First we recall that the symbol calculus for the twisted pseudodifferential
operators has the same formal properties as in the untwisted case -- it is
the smoothing part which is restricted by support conditions. Thus the
standard proofs of the existence of a parametrix, $B,$ for an elliptic
element carry over unchanged. Furthermore the set of parametrices is affine
over the smoothing operators (with restricted support).

First we consider the special case that $E=F.$ The smoothing operators are
only constrained away from the diagonal and the trace functional is
well-defined as usual as the integral of the Schwartz kernel over the diagonal
\begin{equation}
\Tr(A)=\int_X A\big|_{\Diag}.
\label{spallb.42}\end{equation}
When the composition of smoothing operators is defined, 
\begin{equation}
\Tr([A,B])=0
\label{spallb.43}\end{equation}
and this identity extends, by continuity as indicated above, to the case
where one of the operators is a pseudodifferential operator.

Now, from this it follows that the definition, \eqref{spallb.28}, of
$\ind(A)$ is independent of the choice of parametrix $B$ since if $B,$ $B'$
are two parametrices then $B_t=(1-t)B+tB'$ is a family of parametrices and  
\begin{equation}
\frac{d}{dt}\Tr([A,B_t])=\Tr([A,B'-B])=0
\label{spallb.44}\end{equation}
since $B'-B$ is smoothing. 

It is vital to establish the homotopy invariance of this index. To do so we
use the trace-defect formula from \cite{Melrose-Nistor1}; it is very closely
related to the proof of the Atiyah-Patodi-Singer index theorem in
\cite{MR96g:58180}. First we may define the residue trace on
$\lPs{\cL,\epsilon }\bbZ{X;E}$ following the idea of Guillemin
\cite{Guillemin2}. Namely, the residue trace can be defined as the reside
at $z=0$ of the meromorphic function 
\begin{equation}
\rTr(A)=\lim_{z\to0}zF(z),\ F(z)=\Tr(AQ(z))
\label{spallb.45}\end{equation}
provided $Q(z)$ is a family of pseudodifferential operators of complex order
$z$ which is everywhere elliptic and satisfies $Q(0)=\Id.$ Even in the
twisted case we can find such a family. One approach, indicated in
\cite{mms3}, is to take a generalized Laplacian, $L\in\lPs{\cL,\epsilon
}2{X;E},$ construct the singularity of its formal heat kernel and then take
the Laplace transform. The construction of the singularity, at
$\{t=0\}\times\Diag$ of the heat kernel $e^{-tL}$ is known to be completely
local and symbolic (see for example \cite{MR96g:58180} where this is done
in a more general case) and hence can be carried out in the twisted case, up
to a smoothing error term and with support in any preassigned neighbourhood
of the diagonal. The virtue of this construction is that it gives
an entire family of operators $Q(z)$ of complex order $z$ such that 
\begin{multline}
Q(0)=\Id\Mand Q(z)Q(-z)=\Id+R(z),\\ \bbC\ni z\longrightarrow
\lPs{\cL,\epsilon }{-\infty}{X;E}\text{ entire with }R(0)=0,\ R'(0)=0.
\label{spallb.58}\end{multline}
The vanishing of $R'(0)$ follows by direct differentiation of the defining
identity.

The proof that the function $F(z)$ in \eqref{spallb.45} is meromorphic is
again the same as in the standard case, as in the (corrected version of)
the original argument of Seeley (\cite{MR38:6220}) and $F(z)$ has at most a
simple pole at $z=0.$ Furthermore the residue at $z,$ defining $\rTr(A)$ is
independent of the choice of $Q(z)$ since if $Q'(z)$ is another such family
then $Q'(z)-Q(z)=zE(z)$ where $E(z)$ is also an entire family of operators
of complex order $z.$ Since the residue reduces to the same local
computation as in the untwisted case, it is given by the same formula,
namely the integral over the cosphere bundle of the trace of the term of
homogeneity $-n$ in the full symbol expansion. In particular  
\begin{equation}
\rTr(\Id)=0.
\label{spallb.59}\end{equation}

The regularized trace of $A,$  
\begin{equation}
\bTr_Q(A)=\lim_{z\to0}(\Tr(AQ(z))-\rTr(A)/z)
\label{spallb.46}\end{equation}
does depend on the choice of the family $Q(z).$ However once a choice is
made, it gives a functional extending the trace. In the special case that
we are considering where $A\in\lPs{\cL,\epsilon }\bbZ{X;E}$ `acts on a
fixed bundle' this allows the index to be rewritten  
\begin{equation}
\ind(A)=\bTr_Q([A,B]).
\label{spallb.47}\end{equation}
The family $Q$ also defines a derivation on the algebroid
$\lPs{\cL,\epsilon }\bbZ{X;E}.$ Namely 
\begin{equation}
D_QT=\frac{d}{dz}Q(z)TQ(-z)\big|_{z=0}.
\label{spallb.48}\end{equation}
Notice that the family on the right is an entire family of fixed order,
one less than that of $T$ (since $Q(z)$ is principally diagonal) so
\begin{equation}
D_Q:\lPs{\cL,\epsilon }m{X;E}\longrightarrow\lPs{\cL,\epsilon+\delta}m{X;E}
\label{spallb.49}\end{equation}
if $\epsilon$ and $\delta$ are small enough. This is a derivation in the
sense that 
\begin{equation}
D_Q(T_1T_2)=(D_QT_1)T_2+T_1(D_QT_2)
\label{spallb.50}\end{equation}
provided all supports are small enough. In fact $D_Q$ is independent of the
choice of $Q$ up to an interior derivation. Notice that for any
$T\in\lPs{\cL,\epsilon }\bbZ{X;E}$  
\begin{equation}
\rTr(D_QT)=0
\label{spallb.64}\end{equation}
since this is the residue at $\tau=0$ of 
\begin{multline*}
\Tr(D_QTQ(\tau)=\frac{d}{dz}\Tr\left(Q(z)TQ(-z)Q(\tau)\right)\big|_{z=0}\\
=\frac{d}{dz}\Tr\left(TQ(-z)Q(\tau)Q(z)\right)\big|_{z=0}\\
=\frac{d}{dz}\Tr\left(T(Q(\tau)+E(z,\tau)\right)\big|_{z=0}
\label{spallb.65}\end{multline*}
where $W(z,\tau)$ is entire in both variables with values in the smoothing
operators. Thus there is no singularity at $\tau=0.$ 

The trace-defect formula now follows directly. If $T,S\in\lPs{\cL,\epsilon
}\bbZ{X;E}$ with $\epsilon >0$ (and $\delta >0$ from $Q)$ small enough then 
\begin{equation}
\bTr_Q([T,S])=\rTr(SD_QT).
\label{spallb.51}\end{equation}
Indeed using the trace identity when $\Re z<<0,$  
\begin{multline}
\Tr\left([T,S]Q(z)\right)=\Tr\left(SQ(z)T-STQ(z)\right)\\
=\Tr\left(S(Q(z)TQ(-z)-T)Q(z)\right)-\Tr\left(SQ(z)TR(-z)\right).
\label{spallb.60}\end{multline}
The last term here is the trace of an entire family of smoothing operators,
vanishing at $z=0.$ Furthermore, $Q(z)TQ(-z)-T$ is an entire family of
operators of fixed order, vanishing at $z=0$ so can be written
$zD_QT+O(z^2)$ and we arrive at \eqref{spallb.51}.

Now the residue trace vanishes on operators of low order so is a trace on
the symbolic quotient 
\begin{equation}
\lPs{\cL,\epsilon}\bbZ{X;E}/\lPs{\cL,\epsilon }{-\infty}{X;E}
\label{spallb.61}\end{equation}
in which an elliptic operator is, by definition, invertible. Thus the index
of an invertible element $a$ in \eqref{spallb.61} is 
\begin{equation}
\ind(a)=\Tr([A,B])=\rTr(a^{-1}D_Qa)
\label{spallb.62}\end{equation}
which is also independent of the choice of $Q.$ In this case the homotopy
invariance follows directly, since if $a_t$ is an elliptic family depending
smoothly on a parameter $t$ then 
\begin{multline}
\frac{d}{dt}\ind(a_t)=\frac{d}{dt}\rTr(a^{-1}_tD_Qa_t)\\
=\rTr\left(a^{-1}_tD_Q\dot a_t-a^{-1}_t\dot a_ta^{-1}_tD_Qa_t\right)
=\rTr\left(D_Q(\dot a_ta^{-1}_t)\right)=0
\label{spallb.63}\end{multline}
where $\dot a_t$ denotes the $t$-derivative and \eqref{spallb.64} has been used.

This proof only covers directly the case of elliptic operators on a fixed
bundle $E.$ However, for $A\in\lPs{\cL,\epsilon}\bbZ{X;E,F},$ elliptic
between two different bundles, only relatively minor modifications are
required. Namely we need to choose entire families as above, $Q_E(z)$ and
$Q_F(z)$ for the two bundles. The definition of the index is modified to 
\begin{multline}
\ind(A)=\Tr(AB-\Id_F)-\Tr(BA-\Id_E)\\
=\bTr_F(AB)-\bTr_E(BA)-\bTr_F(\Id)+\bTr_E(\Id)
\label{spallb.66}\end{multline}
where $\bTr_E$ and $\bTr_F$ are the regularized traces defined by $Q_E$ and
$Q_F$ on twisted operators on $E$ and $F.$ Independence of choice follows
as before.

To see homotopy invariance we define the operator 
\begin{equation}
D_Q:\lPs{\cL,\epsilon}\bbZ{X;E,F}\longrightarrow \lPs{\cL,\epsilon}\bbZ{X;E,F},\
D_QA=\frac{d}{dz}\big|_{z=0}Q_F(z)AQ_E(-z)
\label{spallb.67}\end{equation}
for any bundles with a fixed choices of the regularizing families. From the
vanishing of $\frac{d}{dz}Q_E(z)Q_E(-z)$ at $z=0$ it follows that these are
again derivations in a module sense. From \eqref{spallb.66} the index is
the regularized value at $z=0$ of
\begin{multline}
\Tr(BQ_F(z)A)-\bTr(BAQ_E(z))-\bTr_F(\Id)+\bTr_E(\Id)\\
=\Tr\left(B(Q_F(z)AQ(-z)-A)Q_E(z)\right)-
\Tr\left(B(Q_F(z)R_F(-z)\right)-\bTr_F(\Id)+\bTr_E(\Id).
\label{spallb.68}\end{multline}
The second term vanishes at $z=0$ and evaluating the first gives 
\begin{equation}
\ind(A)=\rTr(a^{-1}D_Qa)-\bTr_F(\Id)+\bTr_E(\Id)
\label{spallb.69}\end{equation}
where we have replaced $A$ by its image in the quotient by the smoothing
operators. From this the homotopy invariance follows as before.

Not only does the formula \eqref{spallb.69} lead to the homotopy invariance
of the index, but it also shows the multiplicativity. Given three bundles
$E,F,G$ with $A_1\in\lPs{\cL,\epsilon }\bbZ{X;E,F}$ and
$A_2\in\lPs{\cL,\epsilon }\bbZ{X;F,G}$ elliptic with full symbolic images
$a_1,$ $a_2$ we see that 
\begin{multline}
\ind(A_2A_1)=\rTr((a_2a_1)^{-1}D_Q(a_2a_1))-\bTr_G(\Id)+\bTr_E(\Id)\\
=
\rTr((a_2^{-1}D_Qa_2)\rTr((a_1^{-1}D_Qa_1)-\bTr_G(\Id)+\bTr_E(\Id)=\ind(A_2)+\ind(A_1). 
\label{spallb.71}\end{multline}

The index of the direct sum of two operators is trivially the sum of the
indexes, so from \eqref{spallb.71} and the homotopy invariance we conclude
that the index actually defines a group homomorphism from K-theory as in
the untwisted case 
\begin{equation}
K^0(\com{T^*X},S^*X)\longrightarrow\bbR.
\label{spallb.72}\end{equation}
Here, any class in the K-theory of $\com{T^*X},$ the radial
compactification of the cotangent bundle, relative to its bounding sphere
bundle, is represented by an elliptic symbol $a\in\hom(E,F)$ over $S^*X,$
for bundles $E$ and $F$ over $X$ and the discussion above shows that the
index is the same for two representatives of the same class.

Now, from \eqref{spallb.72} we deduce that the map is actually vanishes on
torsion elements of K-theory, \ie is well defined on
$K^0(\com{T^*X},S^*X)\otimes\bbR.$ Thus, to prove the desired formula
\eqref{spallb.31} it suffices to check it on a set of elements of which
span the K-theory, over $\bbR$ (or $\bbQ).$ If $X$ is even-dimensional the
original observation of Atiyah and Singer is that the bundle-twisted
signature operators are enough to do this. This argument applies directly
here and the arguments of \cite{mms3} again apply to show that the local
index theorem for Dirac operators gives the formula in that case, and hence
proves it in general in the even-dimensional case. For the odd-dimensional
case it is enough to suspend with a circle to pass to the even-dimensional
case.
\end{proof}

\section{Star products}

Notice that in the construction of the star product in
Theorem~\ref{spallb.5} only the $\bbS$-invariant part of the Heisenberg,
and Toeplitz, algebra is used. There is a close connection between the
notion of a local line bundle and the $\bbS$-invariance; this is enough to
allow the invariant part of the algebroid to be constructed directly.

\begin{proposition}\label{18.6.2004.1} Let $L$ be an Hermitian line bundle
over a manifold $M$ with $S$ the circle bundle of $L$ then there is a
canonical isomorphism between distributions on $S\times S$ which are invariant
under the conjugation $\bbS$-action and distributions on the circle bundle, $Q,$
of $\pi^*_LL\otimes \pi^*_RL^{-1}=\Hom(L).$
\end{proposition}

\begin{proof} The map from the total product to the exterior tensor product 
\begin{multline}
S\times S\ni(p,\tau;p',\tau')\longrightarrow
(p,p',\tau\otimes(\tau')^{-1})\in Q\\
=\{(p,p',\sigma );(p,p')\in M\times M,\
\sigma \in L_p\otimes L_{p'}^{-1},\ |\sigma |=1\}
\label{18.6.2004.2}\end{multline}
is a circle bundle with fibre action of $\bbS$ given by the conjugation
action on $\pi_L^*L\times \pi_R^*L.$ Thus the invariant distributions on
$S\times S$ are precisely the pull-backs of distributions on $Q.$
\end{proof}

Still in the integral case, under this identification, the $\bbS$-invariant
Heisenberg operators are identified with the space of parabolic conormal
distributions on the `diagonal' 
\begin{equation}
\IHPs mS= I^{m'}(Q;D;\lambda),\
D=\{(p,p,s)\in Q;p=p',\ s=1\}
\label{spallb.6}\end{equation}
with $\lambda$ being the contact form on $Q.$

So, in the general case of a possibly non-integral symplectic form we
simply define 
\begin{equation}
\IHlPs\epsilon mS=I^{m+\frac12}_{\text{c}}(Q_\epsilon ,D,\lambda ;).
\label{spallb.7}\end{equation}

\begin{proposition}\label{spallb.8} The kernels \eqref{spallb.7} form an
  algebroid with composition restricted only by supports: 
\begin{equation}
\IHlPs{\epsilon}mS\circ\IHlPs{\epsilon'}{m'}S\subset
  \IHlPs{\epsilon+\epsilon'}{m+m'}S
\label{spallb.9}\end{equation}
for all sufficiently small $\epsilon,$ $\epsilon'>0.$
\end{proposition}

\begin{proof} This is a local result once the composition formula is written
  down invariantly and therefore follows from the standard theory of
  Heisenberg operators.
\end{proof}

\begin{theorem}\label{spallb.10} For any symplectic manifold there exists
  $P_\epsilon\in\IHlPs\epsilon 0S$ with $P_\epsilon ^2=P_\epsilon$ modulo
  smoothing and $\sigma (P)$ the field of projections for a positive almost
  complex structure on $M$ and the associated invariant Toeplitz algebroid 
\begin{equation}
\ITplPs\epsilon mM=P_{\epsilon /3}\IHlPs{\epsilon /3}mSP_{\epsilon/3}
\label{spallb.11}\end{equation}
induces a star product on the quotient 
\begin{equation}
\CI(M)[[\rho ]]=\ITplPs\epsilon 0M/\ITplPs\epsilon {-\infty}M.
\label{spallb.12}\end{equation}
\end{theorem}

Again, the global setup having been defined, this is in essence a local
result and hence follows as in the integral case.

Note that only `pure' star products arise directly this way, those
classified by $H^2(M,\bbR).$ As in Fedosov's original construction, one can
pass to the general star product by twisting asymptotic sums of pure star
products.


\def\cprime{$'$} \def\cprime{$'$} \def\cdprime{$''$} \def\cprime{$'$}
  \def\cprime{$'$} \def\ocirc#1{\ifmmode\setbox0=\hbox{$#1$}\dimen0=\ht0
  \advance\dimen0 by1pt\rlap{\hbox to\wd0{\hss\raise\dimen0
  \hbox{\hskip.2em$\scriptscriptstyle\circ$}\hss}}#1\else {\accent"17 #1}\fi}
  \def\cprime{$'$} \def\ocirc#1{\ifmmode\setbox0=\hbox{$#1$}\dimen0=\ht0
  \advance\dimen0 by1pt\rlap{\hbox to\wd0{\hss\raise\dimen0
  \hbox{\hskip.2em$\scriptscriptstyle\circ$}\hss}}#1\else {\accent"17 #1}\fi}
  \def\cprime{$'$} \def\bud{$''$} \def\cprime{$'$} \def\cprime{$'$}
  \def\cprime{$'$} \def\cprime{$'$} \def\cprime{$'$} \def\cprime{$'$}
  \def\cprime{$'$} \def\cprime{$'$} \def\cprime{$'$} \def\cprime{$'$}
  \def\polhk#1{\setbox0=\hbox{#1}{\ooalign{\hidewidth
  \lower1.5ex\hbox{`}\hidewidth\crcr\unhbox0}}} \def\cprime{$'$}
  \def\cprime{$'$} \def\cprime{$'$} \def\cprime{$'$} \def\cprime{$'$}
  \def\cprime{$'$} \def\cprime{$'$} \def\cprime{$'$} \def\cprime{$'$}
  \def\cprime{$'$} \def\cprime{$'$} \def\cprime{$'$} \def\cprime{$'$}
  \def\cprime{$'$} \def\cprime{$'$} \def\cprime{$'$} \def\cprime{$'$}
  \def\cprime{$'$} \def\cprime{$'$} \def\cprime{$'$} \def\cprime{$'$}
  \def\cprime{$'$} \def\cprime{$'$} \def\cprime{$'$} \def\cprime{$'$}
  \def\cprime{$'$} \def\cprime{$'$}
\providecommand{\bysame}{\leavevmode\hbox to3em{\hrulefill}\thinspace}
\providecommand{\MR}{\relax\ifhmode\unskip\space\fi MR }
\providecommand{\MRhref}[2]{%
  \href{http://www.ams.org/mathscinet-getitem?mr=#1}{#2}
}
\providecommand{\href}[2]{#2}

\end{document}